\documentclass[a4]{article}
\usepackage[metapost]{mfpic}
\usepackage{indentfirst,latexsym,amsmath,amssymb}
\usepackage{amsthm}

\renewcommand{\thefigure}{\arabic{figure}}

\newenvironment{mfpfig}[1]
{\figure[htb] \centering \refstepcounter{figure} \label{#1} \par\medskip \endfigure}

\newtheorem{thm}{Theorem}[section] \newtheorem{prop}[thm]{Proposition}
\newtheorem{lm}[thm]{Lemma}
\newtheorem{hyp}[thm]{Hypothesis}
 
 \newtheorem{rem}[thm]{Remark}
\newtheorem{ex}[thm]{Example} 
 \newcounter{i} 
\newcommand{\rr}{\mathbb R}
\newcommand{\rn}{{\mathbb R}^n}
 \newcommand{\xx}{\dot{x}}
\newcommand{\xxx}{\ddot{x}} \newcommand{\eps}{\varepsilon}

\newcommand{\grad}{\mbox{\rm{grad}}}
\newcommand{\V}{V_T}
\newcommand{\dg}{\mbox{\rm{deg}}}
\newcommand{\dist}{\mbox{\rm{dist}}}

\opengraphsfile{myfigs}

\title{On the Dirichlet problem in billiard spaces}

\author{Grzegorz Gabor\\ Faculty of Mathematics and Computer Science,\\ Nicolaus
Copernicus University,\\ Chopina 12/18, 87-100 Toru\'{n}, Poland\\ ggabor@mat.umk.pl}

\date{}

\begin{document}

\maketitle

\begin{abstract}
The constrained Dirichlet boundary value problem $\xxx=f(t,x)$, $x(0)=x(T)$, is studied in billiard spaces, where impacts occur in boundary points. Therefore we develop the research on impulsive Dirichlet problems with state-dependent impulses. Inspiring simple examples lead to an approach enabling to obtain both the existence and multiplicity results in one dimensional billiards. Several observations concerning the multidimensional case are also given.

\end{abstract}


\renewcommand{\thefootnote}{\arabic{footnote}}



\section{Introduction}\label{intro}

For three hundred and fifty years several important problems of mechanics and physics have been investigated and modeled as mathematical problems with impacts. Starting with the works of G. D. Birkhoff, dynamical systems in spaces of the billiard type have been intensively studied. The simplest impact law, for absolutely elastic impacts, can be described geometrically as the equality of angles before and after a collision with a boundary of the billiard space. This law, for simplicity, will be assumed in the present paper.

An elementary observation is that the dynamical system of a billiard type with a uniform motion can be modeled by the simple impulsive second-order system
\begin{equation}\label{birk}
 \left\{ \begin{array}{ll}
\xxx(t)=0, & \mbox{for a.e. } t\geq 0,\\
\xx(s+)=\xx(s)+I(x(s),\xx(s)), & \mbox{if } x(s)\in \partial K,
\end{array}\right.
\end{equation}
where $K=\overline{int K}\subset\rn$ is a compact subset, and $I$ is an impulse function describing the impact law. It is easy to check that for a unit ball $B(0,1)\subset\rn$ and for the equality of the angle of incidence and angle of reflection, one has $I(x(s),\xx(s))=-2\langle x(s),\xx(s)\rangle x(s)$.

Let us imagine a one-dimensional billiard which is not a straight line but a graph of some differential function $\gamma:[a,b]\to\mathbb R$. We can think about some hills and valleys on our simple one-dimensional table. Assume the gravity directed downstairs. Then, the horizontal component of the acceleration is nonzero. In fact, the motion can be described by a more general equation $\xxx(t)=-k\,\grad \gamma(x(t))$, where $k$ is some constant depending on the gravity.

If we allow an external force depending on time (e.g., a wind), we get even more general equation $\xxx(t)=f(t,x(t))$. One can also easily guess that for tables generated by a nondifferentiable but Lipschitz function $\gamma$ we obtain a second-order differential inclusion $\xxx(t)\in F(t,x(t))$. This all above motivates to investigate the system
\begin{equation}\label{problem}
 \left\{ \begin{array}{ll}
\xxx(t)=f(t,x(t)), & \mbox{for a.e. } t\geq 0, x(t)\in int K\\
\xx(s+)=\xx(s)+I(x(s),\xx(s)), & \mbox{if } x(s)\in \partial K.
\end{array}\right.
\end{equation}

Different kinds of problems in billiards are interesting, e.g., existence of  periodic motions and their stability, number of distinct periodic trajectories, etc (see \cite{kozlov} and references therein). We are interested in the boundary value problem of the Dirichlet type: [\eqref{problem} and $x(0)=x(T)=0$] motivated by the research explained below.

In last decades a lot of papers on impulsive boundary value problems have been published. Most of them concerns impulses at fixed moments. In this case one can see direct and clear analogies with the approach and results for problems without impulses. Several difficulties appear when impulses depend on the state variable or both the time and state. One meets the case in e.g. differential population models or mechanics where impulses occurs if some quantities attain a suitable barrier. The papers dealing with state-dependent impulsive problems focus attention mainly on initial or periodic problems, and results on the existence, asymptotic behavior or a stability of solutions. We refer e.g. to \cite{basi} (and references therein), where the impulsive periodic problem, also for some kinds of second-order differential equations, is studied.

Unfortunately, the Dirichlet impulsive boundary value problem cannot be brought to the first-order one, and different techniques are needed (comp. some recent papers using a variational approach for problems with fixed impulse times, \cite{nieto, nie-ore, zhang}.  Quite recently in \cite{rato} the authors examined the problem

\begin{equation}\label{problem-rato}
 \left\{ \begin{array}{ll}
\xxx(t)=f(t,x(t)), & \mbox{for a.e. } t\in [0,T],\\
\xx(s+)=\xx(s)+I(x(s)), & \mbox{if } s=g(x(s)),
\end{array}\right.
\end{equation}
where $g$ is a $C^1$ function satisfying some additional conditions. They provided a new method to solve the problem. Namely, they successfully transformed the problem to the fixed point problem in an appropriate function space. Note that both the impulse function $I$ and the barrier $Graph(g)=\{(x,s); s=g(x)\}$ are not adequate for billiard problems. Indeed, in billiards the impulse depends also on the velocity before the impact. Moreover, the barrier is not a graph of a function with arguments in a phase space. Note also that in \cite{rato} the assumptions insist all trajectories go through the barrier without coming back while in billiards trajectories stay on the same side of the barrier after the impact. The technique presented in \cite{rato} does not work for the Dirichlet problem in a billiard space. This problem, as far as the author knows, is still unexplored.

Therefore the aim of the paper is to study the existence and multiplicity of solutions to the Dirichlet impulsive boundary value problem
\begin{equation}\label{problem1}
 \left\{ \begin{array}{ll}
\xxx(t)=f(t,x(t)), & \mbox{for a.e. } t\in [0,T], x(t)\in int\, K,\\
\xx(s+)=\xx(s)+I(x(s),\xx(s)), & \mbox{if } x(s)\in \partial K,\\
x(0)=x(T)=0, &
\end{array}\right.
\end{equation}
where $I$ describes the impact law of
\renewcommand{\labelenumi}{(H0)}
\begin{enumerate}
\item the equality of the angle of incidence and angle of reflection and the equality of a length of the velocity vector before and after the impact.\label{I}
\end{enumerate}
 By a solution of \eqref{problem1} we mean a continuous function $x:[0,T]\to\rn$, which has an absolutely continuous derivative in intervals where $x(t)\in int\, K$, satisfies $\xxx(t)=f(t,x(t))$ for a.e. $t\in [0,T]$, the impulse condition $\xx(s+)=\xx(s)+I(x(s),\xx(s))$ in boundary points, and the boundary value condition $x(0)=x(T)=0$. Note that different solutions of problem \eqref{problem} can have different numbers of impact points. In general they also can slide along the barrier $\partial K$, and their integral representation is difficult. So, the problem attracts by its nontriviality.

We start in Section \ref{2ex} with two examples inspiring the further research. Then the results about a one-dimensional case are presented in Section \ref{1dim}. The final section concerns problems in $\rn$ so in multidimensional billiard spaces. The author hopes the paper is a valuable development of the area of second-order impulsive boundary value problems. Simultaneously, the theory of billiards is developed in the direction of crooked tables.


\section{Two inspiring examples} \label{2ex}

Let $K=[-a,a]$, where $a>0$, and let $G$ be the Green function for the autonomous problem $\xxx=0, x(0)=x(T)=0$, see e.g. \cite{stak}. We assume in the whole paper that
\renewcommand{\labelenumi}{(H1)}
\begin{enumerate}
\item the right-hand side $f$ (comp. \eqref{problem1}) is a Carath\'eodory function, shortly $f\in Car([0,T]\times\mathbb R)$, which is, for simplicity, integrably bounded, i.e., $f(\cdot,x):[0,T]\to \mathbb R$ is measurable for every $x\in \mathbb R$, $f(t,\cdot):\mathbb R\to \mathbb R$ is continuous for a.e. $t\in [0,T]$, and $|f(t,x)|\leq m(t)$ for some integrable function $m\in L^1([0,T])$ and each $(t,x)\in [0,T]\times \mathbb R$.\label{F}
\end{enumerate}
 Then solutions of the nonimpulsive Dirichlet problem can be obtained as fixed points of the operator $\mathcal F$ in $C^1([0,T])$,
$$\mathcal F(x)(t):=\int_0^T G(t,s)f(s,x(s))ds.$$
Assume that $x\in Fix \mathcal F$. Then $|x(t)|\leq M$ for some $M\geq 0$ and every $t\in [0,T]$. If $M\leq a$, then the solution $x$ remains in $int K$ and the problem \eqref{problem1} is trivially solved. If $M>a$, then the trajectory bounces off the barrier, and the situation becomes nontrivial.

In the first example below we consider simple autonomous equation.

\begin{ex} \label{ex1}
\begin{em}
Consider the problem
\begin{equation}\label{problem-ex1}
 \left\{ \begin{array}{ll}
\xxx(t)=2, & \mbox{for a.e. } t\in [0,1], |x(t)|<\frac{1}{8},\\
\xx(s+)=-\xx(s), & \mbox{if } |x(s)|=\frac{1}{8},\\
x(0)=x(1)=0. &
\end{array}\right.
\end{equation}
It means that $I(x(s),\xx(s))=-2\xx(s)$.

Notice that the unique solution of the corresponding nonimpulsive problem takes the form $x(t)=t(t-1)$ with $||x||=1/4$. Therefore, it does not fit in $K=[-1/8,1/8]$. The first time when the trajectory meets the barrier can be easily computed as $t_0=\frac{2-\sqrt{2}}{4}$. After the impact, we have a trajectory $x(t)=t^2+(\sqrt{2}-1)t+\frac{1-\sqrt{2}}{2}$. One can check that, after some impacts, we obtain $x(1)\neq 0$. To find a solution of \eqref{problem-ex1} we notice, at first, that the unique solution of the autonomous nonimpulsive problem
\begin{equation}\label{problem-ex1-1}
 \left\{ \begin{array}{ll}
\xxx(t)=f(x(t)), & \mbox{for } t\in [0,1],\\
x(0)=x(1)=0 &
\end{array}\right.
\end{equation}
satisfies the equality $x(1/2-t)=x(1/2+t)$ for every $t\in [0,1/2]$. Indeed, one can check that $y(t):=x(1-t)$ is a solution, and, by the uniqueness, $x(t)=x(1-t)$.

Now, we would like to find an initial velocity $v=\xx(0)$ such that the function starting as $x(t):= t(t+v)$ has the following moments of impacts: $\{t_1,1/2,1-t_1\}$ and $x(1)=0$, so we look for a solution of \eqref{problem-ex1} symmetric with respect to $t=1/2$. It is sufficient to find $v$ such that $4t_1+2t_0=1$, where $t_0$ is a negative solution of the equation $t(t+v)=1/8$ (see Figure \ref{fig1}). After some computation we obtain $v\approx -0.8568$, $t_0\approx -0.127$ and $t_1\approx 0.186475$.

\begin{center}
\setlength{\mfpicunit}{5cm}
\begin{mfpfig}{fig1}
\begin{mfpic}{-.2}{1.2}{-.24}{.24}

\axes \xmarks{-.127,.186,.5,.814,1} \ymarks{-.125,.125} \tlpointsep{3pt} \axislabels
x{{$0$}-.04,{$1$}1.02,{$t_1$}.17,{$1-t_1$}.8} \axislabels
y{{$-\frac{1}{8}$}-.18,{$\frac{1}{8}$}.18}
 \tlabel[bl](.02,.17){$x$}
 \tlabel[bl](1.1,-.1){$t$}
\tlabel[bl](-.18,-.1){$t_0$}
\dashed\polyline{(-.13,.125),(1,.125)}
\dashed\polyline{(-.13,-.125),(1,-.125)}
\dashed\polyline{(1,.125),(1,-.125)}
\penwd{3pt} \draw[black]\function{0,.186475,.1}{x**2-.8568*x}
\draw[black]\function{.186475,.5,.1}{(x+.48385)**2-.8568*(x+.48385)}
\draw[black]\function{.5,.813525,.1}{(x-.62705)**2-.8568*(x-.62705)}
\draw[black]\function{.813525,1,.1}{(x-.1432)**2-.8568*(x-.1432)}
 \penwd{1pt}
\dotted\draw[black]\function{-.127,0,.05}{x**2-.8568*x}
 \tcaption{Figure \thefigure. Solution of \eqref{problem-ex1} for $\xx(0)\approx -0.8568$.}
\end{mfpic}
\end{mfpfig}

\end{center}

The solution of problem \eqref{problem-ex1} is not unique. To find the second one it is sufficient to solve the equation $8t_1-6t_0=1$ (we obtain the initial velocity $v_1\approx -1.76579$). It has $7$ impact points. In fact, there are infinitely many solutions obtained when the absolute value of the initial velocity tends to infinity.
\end{em}
\end{ex}

The second example concerns a nonautonomous case, so we loose a symmetry with respect to $t=T/2$.

\begin{ex} \label{ex2}
\begin{em}
Consider the problem
\begin{equation}\label{problem-ex2}
 \left\{ \begin{array}{ll}
\xxx(t)=6t, & \mbox{for a.e. } t\in [0,1], |x(t)|<\frac{3}{8},\\
\xx(s+)=-\xx(s), & \mbox{if } |x(s)|=\frac{3}{8},\\
x(0)=x(1)=0. &
\end{array}\right.
\end{equation}

Now, the function $x(t)=t^3-t$ is the unique solution of the corresponding nonimpulsive problem. But the barrier is such that $x(1/2)=-3/8$ and $\xx(1/2)<0$. Hence, we have an impact. In fact, it is the only impact before $t=1$. Denote by $x_1$ the function which equals $x$ up to the impact, $\xxx_1(t)=6t$ for $t\geq t_1=1/2$, and $\lim_{s\downarrow 1/2}\xx_1(s)=-\xx_1(1/2)$. One can check that $x_1(t)=t^3-\frac{1}{2}t-\frac{1}{4}$ for $t>1/2$. Therefore, $x_1(1)=1/4$ (see Figure \ref{fig2}).

\begin{center}
\setlength{\mfpicunit}{5cm}
\begin{mfpfig}{fig2}
\begin{mfpic}{-.2}{1.2}{-.44}{.44}

\axes \xmarks{.5} \ymarks{-.375,.375} \tlpointsep{3pt} \axislabels
x{{$0$}-.04,{$1$}1.02,{$\frac{1}{2}$}.5} \axislabels
y{{$-\frac{3}{8}$}-.4,{$\frac{3}{8}$}.37}
 \tlabel[bl](.02,.42){$x$}
 \tlabel[bl](1.1,-.1){$t$}
\dashed\polyline{(0,.375),(1,.375)}
\dashed\polyline{(0,-.375),(1,-.375)}
\dashed\polyline{(1,.375),(1,-.375)}
\draw[black]\function{0,1,.1}{x**3-x}
\penwd{3pt} \draw[black]\function{0,.5,.1}{x**3-x}
\draw[black]\function{.5,1,.1}{x**3-.5*x-.25}
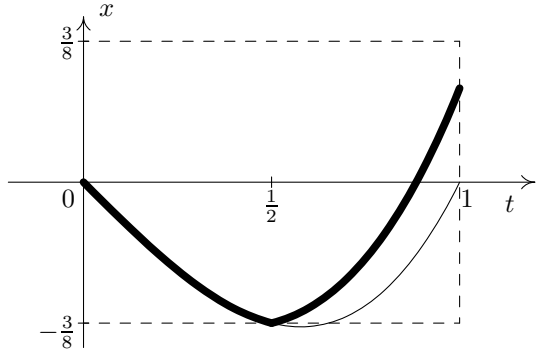
 \tcaption[3,2]{Figure \thefigure. Solution of $\xxx(t)=6t$, $x(0)=x(1)=0$, and the trajectory with the impulse effect.}
\end{mfpic}
\end{mfpfig}

\end{center}
The initial velocity of $x_1$ is equal to $\xx_1(0)=-1$. We know that each solution $x_v$, depending on the initial velocity $\xx_v(0)=v$, of the impulsive Cauchy problem
\begin{equation}\label{problem-ex2-1}
 \left\{ \begin{array}{ll}
\xxx(t)=6t, & \mbox{for a.e. } t\in [0,1], |x(t)|<\frac{3}{8},\\
\xx(s+)=-\xx(s), & \mbox{if } |x(s)|=\frac{3}{8},\\
x(0)=0 &
\end{array}\right.
\end{equation}
up to the first impact has the form $x_v(t)=t^3+vt$. It is not hard to check that for every $v\in [-1,1]$ we get $x_v(1)>0$, and for $-\sqrt[3]{243/256}<v\leq 1$ the trajectory meets only the upper part of the barrier ($x=3/8$) (see Figure \ref{fig3}).

\begin{center}
\setlength{\mfpicunit}{5cm}
\begin{mfpfig}{fig3}
\begin{mfpic}{-.2}{1.2}{-.44}{.44}

\axes \xmarks{.5,.7211} \ymarks{-.375,.375} \tlpointsep{3pt} \axislabels
x{{$0$}-.04,{$1$}1.02,{$\frac{1}{2}$}.5} \axislabels
y{{$-\frac{3}{8}$}-.4,{$\frac{3}{8}$}.37}
 \tlabel[bl](.02,.42){$x$}
 \tlabel[bl](1.1,-.1){$t$}
\dashed\polyline{(0,.375),(1,.375)}
\dashed\polyline{(0,-.375),(1,-.375)}
\dashed\polyline{(1,.375),(1,-.375)}
\penwd{3pt}
\draw[black]\function{0,.7211248,.1}{x**3}
\draw[black]\function{.7211248,1,.1}{x**3-6*.52*x+2.25}
\penwd{1pt}
\draw[black]\function{0,.3367965291,.1}{x**3+x}
\draw[black]\function{.3367965291,1,.1}{x**3-1.680591412*x+0.9028138835}
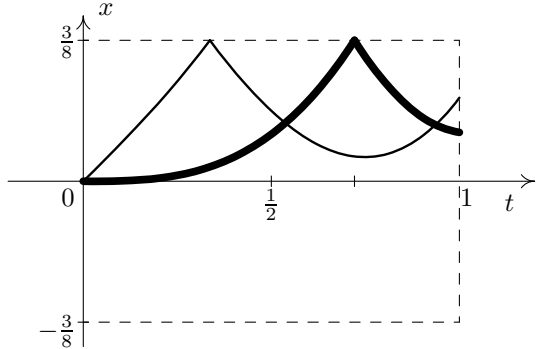
 \tcaption{Figure \thefigure. Trajectory for $v=0$ (bold) and for $v=1$.}
\end{mfpic}
\end{mfpfig}

\end{center}

Taking into account the above trajectories, we try to solve problem \eqref{problem-ex2} by increasing the absolute value of $v$.

Assume that $v<-1$. The first impact time $t_1$ is a solution of $t_1^3=-vt_1-3/8$. After this impact the trajectory is given by $x(t)=t^3-(6t_1^2+v)t-4vt_1-9/4$ (we use an impact law in $t_1$). Analogously, after the second impact time $t_2$ we obtain $x(t)=t^3-(6t_2^2-6t_1^2-v)t+3/8+5t_2^3-6t_1^2t_2-vt_2$. Now we evaluate in $t=1$ and ask for $v$ such that $x(1)<0$. Since the last point $x(1)$ depends continuously on $v$ (see details in Section \ref{1dim}), it follows that there exists a trajectory with $x(1)=0$, so a solution of \eqref{problem-ex2}. One can check that for $v=-1.218$ we have $x(1)\approx -0.0006379$.

\end{em}
\end{ex}

The above two examples show that we can expect solutions of \eqref{problem1} for big initial velocities, and we can also expect a multiplicity of solutions.


\section{Existence and multiplicity of solutions} \label{1dim}

We start with a simple consequence of assumption (H1).

\begin{prop}\label{prop1}
Assume that $x:[0,T]\to K=[-a,a]$ is a solution of the Cauchy problem
\begin{equation}\label{problem-cauchy}
 \left\{ \begin{array}{ll}
\xxx(t)=f(t,x(t)), & \mbox{for a.e. } t\geq 0,\\
\xx(s+)=\xx(s)+I(x(s),\xx(s)), & \mbox{if } x(s)\in \partial K,\\
x(0)=0,
\end{array}\right.
\end{equation}
with an initial velocity $\xx(0)=v$, and $t_1,t_2,\ldots$ the moments of impacts. Then, for each sufficiently big $|v|$, one has $x(t_k)x(t_{k+1})<0$ for every $k\geq 1$, that is, each two consecutive impact points are on the opposite components of the barrier $\partial K$.
\end{prop}

\begin{proof}
It is sufficient to check that $|\xx(t)|\geq d$ for some $d>0$ and every $t\in [0,T]$, if $v$ is sufficiently big. But we know, that
\[\xx(t)=\xx(0)+\int_0^t\xxx(s) ds=v+\int_0^tf(s,x(s))ds,\]
up to the first impact time $t_1>0$, so
\[|\xx(t)|\geq |v|-\int_0^t|f(s,x(s))|ds\geq |v|-\int_0^t m(s)ds.\]
Furthermore, $|\xx(t_1+)|=|\xx(t_1)|$, and
\[\xx(t)=\xx(t_1+)+\int_{t_1}^t f(s,x(s))ds\]
for $t\in (t_1,t_2]$, where $t_2$ is the second impact time. Hence, for $t\in [0,t_2]$,
\[|\xx(t)|\geq |\xx(t_1+)|-\int_{t_1}^t m(s)ds\geq |v|-\int_0^t m(s)ds.\]
We proceed up to $T$, and obtain
\[|\xx(t)|\geq d:=|v|-||m||_1>0,\]
for every $t\in [0,T]$ and $|v|>||m||_1$, where $||m||_1$ is an $L^1$-norm of $m$.
\end{proof}

To simplify our considerations, in the rest of the section we will assume that
\renewcommand{\labelenumi}{(H2)}
\begin{enumerate}
\item the right-hand side $f$ is Lipschitz with respect to the second variable, i.e.,
 \[|f(t,x)-f(t,y)|\leq \alpha(t)|x-y| \mbox{ for a.e. $t\in [0,T]$ and some $\alpha\in L^1([0,T])$}.\]
\end{enumerate}

We are in a position to formulate the following

\begin{thm}\label{thm-dim1}
Let $f$ satisfy {\em (H1)} and {\em (H2)}, and let $x_v, x_w$ be two solutions of \eqref{problem-cauchy} with initial velocities $v, w>||m||_1$, and impulse times $t_1,\ldots,t_k$ and $s_1,\ldots,s_{k+j}$, respectively. If $j\geq 2$, then there exists a solution of \eqref{problem1}.
\end{thm}

The key point in the proof of this theorem can be formulated as

\begin{lm}\label{cont}
Let $f$ satisfy {\em (H1)} and {\em (H2)}, and let $b>||m||_1$. Then the operator $\V:\rr\to [-r,r]$, $\V(v):=x_v(T)$, is continuous in $b$.
\end{lm}

\begin{proof}
Assume that $x_b$ has $k$ impulse times $t_1(b),\ldots,t_k(b)$, and $t_k(b)\neq T$, i.e., $x_b(T)\in (-r,r)$. The case $|x_b(T)|=r$, very similar, is considered at the end of the proof. In the sequel, by $x_{(s,p,v)}$ we will denote the unique solution of the nonimpulsive Cauchy problem
\begin{equation}\label{cauchy1}
 \left\{ \begin{array}{ll}
\xxx(t)=f(t,x(t)), & \mbox{for } t\in [0,T],\\
x(s)=p,\\
\xx(s)=v.
\end{array}\right.
\end{equation}

\noindent {\em Step 1.} We show that
\begin{equation}\label{aim1}
\begin{array}{l}
\text{for every $\eta>0$ there exists $\delta>0$ such that, for $b'\in (b-\delta,b+\delta)$,}\\
\text{the solution $x_{b'}$ has exactly $k$ impacts, and $|t_i(b')-t_i(b)|<\eta$ for any}\\
\text{$i\in \{1,\ldots,k\}$.}
\end{array}
\end{equation}

\noindent Take any $\eta>0$ so small that $(t_i(b)-\eta,t_i(b)+\eta)\cap (t_{i+1}(b)-\eta,t_{i+1}(b)+\eta)=\emptyset$ for $i=1,\ldots,k-1$, and assume, without any loss of generality, that $x_b(t_k(b))=-r$. Then (see Proposition \ref{prop1}) $\xx_b(t)<0$ for every $t\in (t_{k-1}(b),t_k(b))$, and $x_b(t_{k-1}(b))=r$. Since $\xx_b(t_{k-1}(b)+)<0$, one has $x_{(t_{k-1}(b),r,\xx(t_{k-1}(b)+))}(t)<-r$ for any $t>t_k(b)$ close to $t_k(b)$. Take $\bar t\in (t_k(b),t_k(b)+\eta)$ such that $x_{(t_{k-1}(b),r,\xx(t_{k-1}(b)+))}(\bar t)<-r$. By the continuous dependence on initial conditions for \eqref{cauchy1}, there exists $0<\delta_1(k)<1$ such that the inequalities
\begin{equation}\label{ine-k}
|s-t_{k-1}(b)|<\delta_1(k), \ \ |p-r|<\delta_1(k), \ \ |v-\xx_b(t_{k-1}(b)+)|<\delta_1(k)
\end{equation}
imply that the solution $x_{(s,p,v)}$ satisfies $x_{(s,p,v)}(\bar t)<-r$. Hence, it has an impact in some $t_k<\bar t<t_k(b)+\eta$.

On the other hand, $x_b([t_{k-1}(b),t_k(b)-\eta])\subset (-r,r]$ which implies that there exists $\delta_2(k)>0$ such that $\delta_2(k)\leq \delta_1(k)$ and inequalities
\eqref{ine-k} imply that $||x_{(s,p,v)}-x_{(t_{k-1}(b),r,\xx_b(t_{k-1}(b)+))}||_{C^1([0,T])}<\eta$ and $x_{(s,p,v)}([t_{k-1}(b),t_k(b)-\eta])\subset (-r,\infty)$. Thus $t_k>t_k(b)-\eta$.

Now we start to analyze the behavior of solutions near the point $(t_{k-1}(b),r)$. Choose $0<\delta_3(k)<1$ such that
\begin{equation}\label{delta3}
\delta_3(k)<\frac{\delta_2(k)}{2(||\xx_b||+1)} \ \mbox{ and } \ 2\int_{t_{k-1}(b)-\delta_3(k)}^{t_{k-1}(b)+\delta_3(k)}m(s) ds<\frac{\delta_2(k)}{2}.
\end{equation}
Analogously as above we find $\delta_2(k-1)>0$ such that the inequalities
\begin{equation}\label{ine-k-1}
|s-t_{k-2}(b)|<\delta_2(k-1), \ \ |p+r|<\delta_2(k-1), \ \ |v-\xx_b(t_{k-2}(b)+)|<\delta_2(k-1)
\end{equation}
imply that $||x_{(s,p,v)}-x_{(t_{k-2}(b),-r,\xx_b(t_{k-2}(b)+))}||_{C^1([0,T])}<\delta_3(k)$ and $x_{(s,p,v)}$ has an impact in some $t_{k-1}\in (t_{k-1}(b)-\delta_3(k),t_{k-1}(b)+\delta_3(k))$. Then, after the impact, we obtain a function denoted by $\tilde x$ such that $\dot{\tilde x}(t)<0$  and, since $\delta_3(k)<1$,
\[|\dot{\tilde x}(\tau)|\leq |\xx_b(\tau)|+1\leq ||\xx_b||+1,\]
for every $\tau\in (t_{k-1},t_{k-1}(b)+\delta_3(k))$.

Consequently,
\begin{eqnarray*}
r&\geq &\tilde x(t_{k-1}(b)+\delta_3(k))=\tilde x(t_{k-1})+\int_{t_{k-1}}^{t_{k-1}(b)+\delta_3(k)} \dot{\tilde x}(\tau) d\tau\\
&>& r- \int_{t_{k-1}}^{t_{k-1}(b)+\delta_3(k)} |\dot{\tilde x}(\tau)| d\tau>r-2\delta_3(k)(||\xx_b||+1)>r-\delta_2(k).
\end{eqnarray*}
Moreover, assuming without any loss of generality that $t_{k-1}\leq t_{k-1}(b)$,
\begin{eqnarray*}
\lefteqn{|\dot{\tilde x}(t_{k-1}(b)+\delta_3(k))-\xx_b(t_{k-1}(b)+\delta_3(k))|=}\\
& &=\left|\dot{\tilde x}(t_{k-1}+)+\int_{t_{k-1}}^{t_{k-1}(b)+\delta_3(k)} f(\tau,\tilde x(\tau)) d\tau\right.\\
& &\left.-\xx_b(t_{k-1}(b)+)-\int_{t_{k-1}(b)}^{t_{k-1}(b)+\delta_3(k)} f(\tau,x_b(\tau)) d\tau\right|\\
& &=\left|-\dot{\tilde x}(t_{k-1})+\int_{t_{k-1}}^{t_{k-1}(b)+\delta_3(k)} f(\tau,\tilde x(\tau)) d\tau\right.\\
& &\left.+\xx_b(t_{k-1}(b))-\int_{t_{k-1}(b)}^{t_{k-1}(b)+\delta_3(k)} f(\tau,x_b(\tau)) d\tau\right|\\
& &=\left|-\dot{\tilde x}(t_{k-1})+\int_{t_{k-1}}^{t_{k-1}(b)+\delta_3(k)} f(\tau,\tilde x(\tau)) d\tau\right.\\
& &\left.+\xx_b(t_{k-1})+\int_{t_{k-1}}^{t_{k-1}(b)} f(\tau,x_b(\tau)) d\tau-\int_{t_{k-1}(b)}^{t_{k-1}(b)+\delta_3(k)} f(\tau,x_b(\tau)) d\tau\right|\\
& &\leq |\dot{\tilde x}(t_{k-1})-\xx_b(t_{k-1})|+\int_{t_{k-1}}^{t_{k-1}(b)}|f(\tau,\tilde x(\tau)) +  f(\tau,x_b(\tau))| d\tau\\
& &+ \int_{t_{k-1}(b)}^{t_{k-1}(b)+\delta_3(k)} |f(\tau,\tilde x(\tau)) -  f(\tau,x_b(\tau))| d\tau\\
& &\leq |\dot{\tilde x}(t_{k-1})-\xx_b(t_{k-1})|+2\int_{t_{k-1}}^{t_{k-1}(b)+\delta_3(k)}m(\tau) d\tau<\frac{\delta_2(k)}{2}+\frac{\delta_2(k)}{2}=\delta_2(k).
\end{eqnarray*}
We proceed, and find, at last, $\delta:=\delta_2(1)\in (0,1)$ such that $||x_{(0,0,b')}-x_{(0,0,b)}||_{C^1([0,T])}<\delta_3(2)$ for any $b'\in (b-\delta_2(1), b+\delta_2(1))$, and $x_{(0,0,b')}$ has the first impact $t_1$ in $(t_1(b)-\delta_3(2),t_1(b)+\delta_3(2))$, where
\[\delta_3(k)<\frac{\delta_2(k)}{2(||\xx_b||+1)} \ \mbox{ and } \ 2\int_{t_{k-1}(b)-\delta_3(k)}^{t_{k-1}(b)+\delta_3(k)}m(s) ds<\frac{\delta_2(k)}{2}.
\]
Then, for $b'\in (b-\delta_2(1), b+\delta_2(1))$, the solution $x_{b'}$ of problem \eqref{problem-cauchy} has exactly $k$ impacts, and $|t_i(b')-t_i(b)|<\eta$ for any $i\in \{1,\ldots,k\}$.

A proof for $x_b(t_k(b))=r$ is the same.

\noindent {\em Step 2.} Take an arbitrary $\eps>0$. Analogously as in Step 1 we assume that $x_b(t_k(b))=-r$, and we choose $0<\theta<1$ such that the inequalities
\[|s-t_{k}(b)|<\theta, \ \ |p+r|<\theta, \ \ |v-\xx_b(t_{k}(b)+)|<\theta
\]
imply that $||x_{(s,p,v)}-x_{(t_k(b),-r,\xx_b(t_k(b)+))}||<\eps$. In consequence, $||x_{(s,p,v)}(T)-x_{(t_k(b),-r,\xx_b(t_k(b)+))}(T)||<\eps$.

Choose $0<\eta<\theta$ such that
\[\eta<\frac{\theta}{2(||\xx_b||+1)} \ \mbox{ and } \ 2\int_{t_{k}(b)-\eta}^{t_{k}(b)+\eta}m(s) ds<\frac{\theta}{2},
\]
and apply \eqref{aim1} to find $\delta>0$. Now, like in Step 1, we can check that, for every $b'\in (b-\delta,b+\delta)$, the solution $x_{b'}$ satisfies
\[|t_k(b')-t_k(b)|<\theta, \ \ |\xx_{b'}(t_k(b')+)-x_b(t_k(b)+)|<\theta.\]
Thus $|x_{b'}(T)-x_b(T)|<\eps$.\vspace{2mm}

To finish the proof of the lemma it is sufficient to notice that the case $t_k(b)= T$ is quite similar. The only difference is that, for $b'$ close to $b$, the solution $x_{b'}$ can have only $k-1$ impulses. But still this solution considered on a little larger interval, e.g. $[0,T+1]$, must have an impulse near $T$, and the above proof arguments work.
\end{proof}

\begin{rem}
\begin{em}
One can prove the continuity of $V_T$ without the assumption $b>||m||_1$, because of the continuity of the impulse function $I(z)=-2z$, but a formulation of Lemma \ref{cont} is sufficient for our considerations.
\end{em}
\end{rem}

Now we are able to prove Theorem \ref{thm-dim1}.

\begin{proof}
  We will show that there exists $u(\lambda)=(1-\lambda)v+\lambda w$ ($\lambda\in [0,1]$) such that the solution $x_{u(\lambda)}$ of \eqref{problem-cauchy} with the initial velocity $u(\lambda)$ solves \eqref{problem1}.

Without any loss of generality, assume that $x_v(t_k)=r$. Take $\lambda\nearrow 1$. Then $x_{u(\lambda)}(T)$ attains the barrier twice (on opposite sides), at least, because new impacts occur on different parts of the barrier (see Proposition \ref{prop1}). From the Darboux theorem, since $x_{u(\lambda)}(T)$ depends continuously on $\lambda$ (see Lemma \ref{cont}), there exists $\lambda\in [0,1]$ such that $x_{u(\lambda)}(T)=0$. The proof is complete.
\end{proof}

We finish our consideration with the observation that, indeed, there are initial velocities $v$ and $w$ satisfying assumptions of the above theorem.

\begin{thm}\label{thm-number}
Let $f$ satisfy {\em (H1)} and {\em (H2)}, $v\neq 0$, and $x_v$ be a solution of \eqref{problem-cauchy} with $k$ impulse times. Then there exists $c>0$ such that the solution $x_{cv}$ of \eqref{problem-cauchy} has at least $k+2$ impulse times.
\end{thm}

\begin{proof}
Let $t_1(v),\ldots,t_k(v)$ be impulse times for $x_v$. Divide the interval $[0,t_1(v)]$ into three  intervals of equal length.

At first, we denote by $\bar{x}_v$ and $\bar{x}_{cv}$ the solutions of a nonimpulsive Cauchy problem $\xxx =f(t,x), x(0)=0$, and show that there is $c_0>0$ such that $|\bar{x}_{cv}(t_2)-\bar{x}_{cv}(t_1)|>2r$ for $c>c_0$ and for each $t_1, t_2\in [0,T]$ with $|t_2-t_1|\geq t_1(v)/3$.

Indeed, if $cv$ is an initial velocity, then
\[
|\xx_{cv}(t)|\geq |\xx_{cv}(0)| - \int_0^t|m(s)| ds\geq c|v|-||m||_1,
\]
for every $t\in [0,T]$, see the proof of Proposition \ref{prop1}. Take $c_1>0$ such that $c_1|v|-||m||_1>0$, and let $c\geq c_1$. Then
\[||\bar{x}_{cv}(t_2)-\bar{x}_{cv}(t_1)|=\int_{t_1}^{t_2} |\xx_{cv}(s)| ds\geq (c|v|-||m||_1)\frac{1}{3}t_1(v).\]
It is sufficient to find $c_0\geq c_1>0$ such that $\frac{1}{3}(c_0|v|-||m||_1)t_1(v)>2r$. It is easy to see that $c_0>\frac{||m||_1}{|v|}+\frac{6r}{|v|t_1(v)}$.

Now, if $c\geq c_0$, then the solution $x_{cv}$ has an impulse time $t_1(cv)$ in $[0,t_1(v)/3)$. Furthermore, we have the second impulse time $t_2(cv)$ in $[t_1(cv),t_1(cv)+t_1(v)/3)$ and the third one $t_3(cv)\in [t_2(cv),t_2(cv)+t_1(v)/3)$. Thus $t_3(cv)<t_1(v)$.

We can choose $c_0$ so big that
\begin{equation}\label{eqthm1}
|\xx_{cv}(t)|>|\xx_v(t)| \ \ \mbox{ for every } t\in [0,T].
\end{equation}
 Indeed, we know that $|\xx_{cv}(t)|\geq c|v|-||m||_1$ for every $t\in [0,T]$. Similarly we check that $|\xx_v(t)|\leq |v|+||m||_1$ for every $t\in [0,T]$. So, it is sufficient to have $c_0|v|-||m||_1>|v|+||m||_1$ which gives $c_0> (2||m||_1+|v|)/|v|$.

Without any loss of generality we can assume that $\xx_{cv}(t)>\xx_v(t)>0$ up to the first impulse time $t_1(cv)$. Obviously,
\[\bar{x}_{cv}(t_1(v))=\int_0^{t_1(v)}\dot{\bar{x}}_{cv}(s) ds > \int_0^{t_1(v)}\xx_{v}(s) ds =r,\]
hence, $x_{cv}(t_1(cv))=r$ (the upper part $x=a$ of the barrier). From the previous considerations we know that $x_{cv}(t_2(cv))=-r$ and $x_{cv}(t_3(cv))=r$. From \eqref{eqthm1} it follows that
\[t_{3+i}(cv)<t_{1+i}(v) \ \ \mbox{ for every } i\geq 0.\]
Indeed, for $i=1$, if $t_4(cv)\leq t_1(v)$, then $t_4(cv)<t_2(v)$. If $t_4(cv)> t_1(v)$, then $x_{cv}(t_1(v))\in (-r,r)$, and $\xx_{cv}(t_1(v)+)$ and $\xx_{v}(t_1(v)+)$ are negative. Moreover, $\xx_{cv}(t)<\xx_{v}(t)$ for $t>t_1(v)$ up to the next impulse time. Thus, again, $t_4(cv)<t_2(v)$. For $i>1$ we proceed analogously. In consequence, $t_{i+2}(cv)<t_i(v)$.

\end{proof}

Besides giving the existence of solutions of \eqref{problem1}, the above theorem implies their multiplicity, as we can see in the following.

\begin{thm}\label{multi}
If $f$ satisfies {\em (H1)} and {\em (H2)}, then there exist infinitely many solutions of \eqref{problem1}.
\end{thm}

\begin{proof}
One can easily modify a proof of Theorem \ref{thm-number} to obtain a solution of problem \eqref{problem-cauchy} with $2k$ impulse times before $t_1(v)$ for each $k\geq 1$. For each $k$ we get at least $k$ different solutions of problem \eqref{problem1}, and the proof is complete.

\end{proof}


\section{Multidimensional billiards. Perspectives of research}

A multidimensional case is much more complicated, and only some special situations are investigated below. We give ideas how one could solve the problem and indicated some difficulties. Some ideas are formulated as open problems. Hence, we treat the section as an impulse to do deeper research. We start, like in the previous section, with an inspiring example.

\begin{ex}\label{mul-ex1}
\begin{em}

Let the billiard space be equal $K=\overline{B(0,r)}\subset\rr^2$, the closed ball with a center $0$ and radius $r>0$. Consider the Dirichlet problem
\begin{equation}\label{mul-ex1-eq1}
 \left\{ \begin{array}{ll}
\xxx(t)=a:=(a_1,a_2), & \mbox{for a.e. } t\in [0,T],\\
\xx(s+)=\xx(s)+I(x(s),\xx(s)), & \mbox{if } x(s)\in \partial K,\\
x(0)=x(T)=0,
\end{array}\right.
\end{equation}
where $I(x(s),\xx(s))=-\frac{2}{r^2}\langle x(s),\xx(s)\rangle x(s)$. Therefore the line $a_2x-a_1y=0$ is invariant under the flow generated by $\xxx(t)=a, x(0)=0, \xx(0)=(\alpha a_1,\alpha a_2)$. It implies that we can solve the problem \eqref{mul-ex1-eq1} as one-dimensional (see Example \ref{ex1}). Hence, there are infinitely many solutions of \eqref{mul-ex1-eq1}.
\end{em}
\end{ex}

The above example leads to the following result for simple billiards (with a uniform motion).

\begin{thm}\label{uniform}
Assume that $K\subset\rn$, $n=2k$, is a compact smooth (by `smooth' we mean, here and in the sequel, at least of class $C^2$) manifold with boundary, $0\in int K$, and $K$ is strongly star-shaped with respect to $0$, i.e., $tx\in int K$ for every $x\in K$ and $t\in [0,1)$ (see, e.g., \cite{dugr}, p. 77 for a definition of star-shaped sets). Then the Dirichlet problem
\begin{equation}\label{uniform-eq1}
 \left\{ \begin{array}{ll}
\xxx(t)=0, & \mbox{for a.e. } t\in [0,T],\\
\xx(s+)=\xx(s)+I(x(s),\xx(s)), & \mbox{if } x(s)\in \partial K,\\
x(0)=x(T)=0,
\end{array}\right.
\end{equation}
with $I$ satisfying the standard impact low (H0) (see Section \ref{intro}) has a nontrivial solution.
\end{thm}

\begin{proof}
It is sufficient to find a ray $\{tz; t\geq 0 \}$, for some $z\in \partial K$, such that the vector $\overrightarrow{0z}$ is perpendicular to the tangent hyperspace to $\partial K$ at $z$. Indeed, if $z$ satisfies this, then the function $x(t)=vt$ with $v:=\frac{2}{T}z$ attains $\partial K$ at $z$ and in the time $T/2$, then changes the velocity to
$$\xx\left(\frac{T}{2}+\right)=\frac{2}{T}z+I\left(z,\frac{2}{T}z\right) =\frac{2}{T}z-\frac{2}{||z||^2}\left\langle z,\frac{2}{T}z\right\rangle=-v,$$
and finally arrives at $0=x(T)$.

To prove the existence of $z$ mentioned above, we take a small ball $B(0,r)\subset int K$ and denote by $R:\partial B(0,r)\to \partial K$ a projection along rays, existing by the strong star-shape of the set $K$. For each point $y\in \partial K$ we take an outer normal vector $n(y)$ to $K$ at $y$ and denote by $N(y)$ its projection onto the tangent hyperspace to $S^{n-1}$ at $y/|y|$. We define the map $g:S^{n-1}\to \rn$, $g(w):=N(R(rw))$.
Thus we have obtained a tangent vector field on the sphere $S^{n-1}$. From the hairy ball theorem (see \cite{dugr}, Cor. 7.4, p. 238) it follows that $g(w)=0$, for some $w\in S^{n-1}$, which means that $z:=R(rw)$ is the vector we required..
\end{proof}

Let us remark that a regularity of $K$ may be slightly weakened. For instance, if $\partial K$ is a $C^{1,1}$ manifold, then the map $N(\cdot)$ is Lipschitz and, moreover, the set $K$ is an absolute neighborhood retract (see, e.g., \cite{gq}, p. 500). The latter property will be used in a proof of Theorem \ref{thm-a}.

Before the second example let us describe the situation we will deal with. Let $K\subset\rn$ be a compact smooth manifold with boundary, and $0\in int K$. Consider problem \eqref{problem1} with $I$ describing the standard impact law (H0) and  $f$ satisfying
\renewcommand{\labelenumi}{(H1)$_n$}
\begin{enumerate}
\item $f:[0,T]\times\rn\to\rn$ is an integrably bounded Carath\'eodory function, i.e., $f(\cdot,x):[0,T]\to \rn$ is measurable for every $x\in \rn$, $f(t,\cdot):\rn\to \rn$ is continuous for a.e. $t\in [0,T]$, and $|f(t,x)|\leq m(t)$ for some integrable function $m\in L^1([0,T])$ and each $(t,x)\in [0,T]\times \rn$,
\end{enumerate}
and (H2). Notice that $I$ can be given explicitly as $I(y,v):=-2\langle v,n(y)\rangle n(y)$, where, for some $s>0$, $y=x(s)\in\partial K$, $v=\xx(s)$ and $n(y)$ is an outer normal vector to $\partial K$ at $y$. Obviously, $I$ is continuous. Smoothness of the set $\partial K$ together with the continuity of $I$ justifies formulating the following

\begin{hyp}\label{cont-n}
The operator $\V:\rn\to K$, $\V(v):=x_v(T)$, where $x_v$ is a solution of the problem \eqref{problem},
 is continuous.
\end{hyp}

On the way to prove it one meets some difficulties. Some solutions of problem \eqref{problem} can slide along the boundary $\partial K$, at least for some time, and then their dynamics is essentially changed, i.e., $\xxx(t)\neq f(t,x(t))$ on some set with nonzero measure. This can happen if $\langle \xx(s), n(x(s))\rangle=0$ in a boundary point $x(s)\in \partial K$. Therefore we have a nontrivial hybrid system, where the dynamics on $\partial K$ can be described not only by $f$ but also by a regularity of $\partial K$ (by a normal vector to $\partial K$). We hope one could describe it explicitly if $K$ is a sublevel set of some smooth function. We leave the proof of Hypothesis \ref{cont-n} to further studies.

In what follows we show how one can use the above hypothesis.

\begin{ex}\label{mul-ex2}
\begin{em}
Assume that $K$, $f$ and $I$ are as above and define, for every $d\geq 0$, an {\em attainable set in} $T$ as $A_d:=\{x_v(T) ; |v|=d\}$. By Hypothesis \ref{cont-n}, it is a hypersurface (not necessarily smooth) in $K$. Let $\Omega_d$ be the area surrounded by $A_d$. Assume that $0\in \Omega_d$ and $0\not\in \Omega_{ad}$ for some $a\geq 0$ (see Figure \ref{fig4}). Then, from the continuity of $\V$ it follows that there exists $c=(1-\lambda)a+\lambda$, $\lambda\in [0,1]$, such that $0\in A_{cd}$. This means that problem \eqref{problem1} has a solution with an initial velocity $|v_0|=cd$.

\begin{center}
\setlength{\mfpicunit}{3cm}
\begin{mfpfig}{fig4}
\begin{mfpic}{-1}{1}{-1}{.8}
\ellipse{(0,0),1,.8}
\curve{(-.8,0),(0,.7),(.2,.2),(.3,-.1),(-.3,-.5),(-.8,0)}
\dashed\curve[1.5]{(-.7,0),(-.6,.6),(-.1,.4),(-.2,0),(-.3,-.6),(-.7,0)}
\dotted\curve[2]{(-.73,0),(-.56,.5),(-.3,.45),(.06,.6),(-.1,.3),(0,0),(-.3,-.56),(-.7,0)}
\point[2pt]{(0,0)}
 \tlabel[bl](.02,0){$0$}
 \tlabel[bl](.32,-.1){$A_d$}
 \tlabel[bl](-.09,.4){$A_{ad}$}
\tlabel[bl](-.03,-.2){$A_{cd}$}
\tlabel[bl](.9,-.4){$K$}

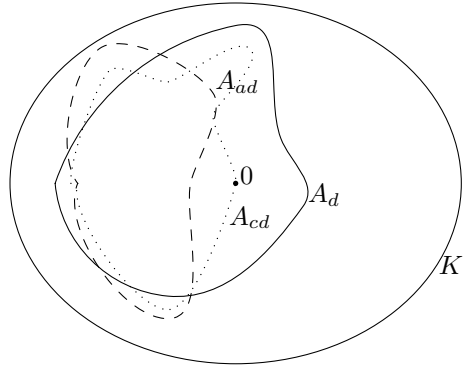
 \tcaption[3,2]{Figure \thefigure. Attainable sets for $|v|$ equal to $d, ad, cd$.}
\end{mfpic}
\end{mfpfig}

\end{center}
\end{em}
\end{ex}

The above example leads us to the idea of using the winding number of a curve, in a 2-dimensional case, or, more generally, a topological degree in $\rn$ (see, e.g., \cite{krza}, p. 6-8, where the degree is presented as the rotation of the vector field).

Indeed, as above we define the sets $A_d=\{\{x_v(T) ; |v|=d\}\subset K\subset\rn$ and $\Omega_d\subset K$, the area surrounded by $A_d$. Assume that $0\not\in A_d$. Obviously, $A_d$ can be treated as the image $A_d=V_T^d(S^{n-1})$, where $V_T^d:S^{n-1}\to \rn$ is defined as $V_T^d(v):=x_{dv}(T)$.
Since $V_T^d$ is continuous, the topological degree $\dg(V_T^d,B(0,1))$ is well defined. Moreover, the family $\{V_T^d\}_{d\geq 0}$ is continuous, by Hypothesis \ref{cont-n}. Thus we get to

\begin{thm}\label{n-exist}
Let $f$ and $I$ satisfy (H1)$_n$-(H2) and (H0), respectively. Assume the truth of Hypothesis \ref{cont-n}, and  that $0\not\in A_d$ for some $d\geq 0$. Then \begin{list} {\em (\roman{i})}{\usecounter{i}}
 \item $\dg(V_T^d,B(0,1))\neq 0$ implies the existence of a solution $x_v$ to problem \eqref{problem1} with an initial velocity $v\in B(0,d)$,
 \item $\dg(V_T^{d_1},B(0,1))\neq \dg(V_T^{d_2},B(0,1))$, for some $d_1<d_2$, implies the existence of a solution $x_v$ to problem \eqref{problem1} with an initial velocity $v\in B(0,d_2)\setminus B(0,d_1)$.
\end{list}
\end{thm}

\begin{proof}
The first part easily follows from the existence property of the degree. The second one from the homotopy invariance. Indeed, if there were no zero of $V_T^d$ for each $d_1\leq d\leq d_2$, then, since $V_T^{d_1}$ and $V_T^{d_2}$ are homotopic, one would have $\dg(V_T^{d_1},B(0,1))=\dg(V_T^{d_1},B(0,1))$; a contradiction.
\end{proof}

A natural hypothesis is that fast movements, i.e., big velocities, can compensate the unevenness of the billiard table described by the nontrivial acceleration $f(t,x)$. To check this, we note the following

\begin{lm}\label{lemma-a}
Under the assumptions (H1)$_n$-(H2) and (H0), for every $\eps>0$ there exists $d>0$ such that, for each $|v|\geq d$, each solution $x_v$ of \eqref{problem}, with an initial velocity $v$, and for each $t\in [0,t_1(v)]$ ($t_1(v)$ is the first time of impact, see Section \ref{1dim}) one has $|x(t)-vt|<\eps$. In other words, the movement is close to the one with a uniform velocity.
\end{lm}

\begin{proof}
Take a small $\rho>0$ such that $\rho ||m||_1<\min\{\eps,1\}$. Put $D:=\delta(K)+1$, where $\delta(K)$ is the diameter of the compact set $K$. Now, let $d>0$ be such that $d\rho>D$. This implies that the curve $z(t)=vt$, for $|v|\geq d$, attains the boundary $\partial K$ before the time $\rho$, and $\dist(z(\rho),\partial K)>1$. Then, for a solution $x_{(0,0,v)}$ of the Cauchy problem $\xxx=f(t,x), x(0)=0, \xx(0)=v$ we obtain
\[|x_{(0,0,v)}(t)-z(t)|\leq \int_0^t|\xx_{(0,0,v)}(s)-v|ds\leq t||m||_1 <\eps\]
for every $t\leq \rho$. Moreover, $|x_{(0,0,v)}(\rho)-z(\rho)|\leq \rho ||m||_1 <1$ which implies that
$x_{(0,0,v)}(\rho)\not\in K$ and, consequently, $x_{(0,0,v)}$ meets the barrier $\partial K$ at a time before $\rho$. This completes the proof.
\end{proof}

The problem is that the increase of $|v|$ causes the increase of a number of impacts. Nevertheless, since the impulse function $I$ is continuous, we can formulate the following hypothesis which is left as an open problem.

\begin{hyp}\label{prop-a}
Under the assumptions (H1)$_n$-(H2) and (H0), for every $\eps>0$ there exists $d>0$ such that for each $|v|\geq d$, each solution $x_v$ of \eqref{problem}, with an initial velocity $v$, and for each $t\in [0,T]$,  one has $|x(t)-vt|<\eps$.
\end{hyp}

Now we are in a position to formulate a result on a correspondence between the general problem \eqref{problem1} and the uniform motion Dirichlet problem \eqref{uniform-eq1}, i.e., the Dirichlet problem in the Birkhoff billiard space.

\begin{thm}\label{thm-a}
We hold the assumptions of Theorem \ref{n-exist}, assume the truth of Hypothesis \ref{prop-a}, and assume that there is a sequence $d_k\nearrow\infty$ such that $\deg(W_T^{d_k},B(0,1))\neq 0$ and $\deg(W_T^{d_{k-1}},B(0,1))\neq \deg(W_T^{d_k},B(0,1))$, where $W_T^d:S^{n-1}\to\rn$, $W_T^d(v):=z_v(T)$ and $z_v$ is a unique solution of the problem \eqref{uniform-eq1} with the initial velocity $v$. Assume also that $\dist(W_T^{d_k}(S^{n-1}),0)\geq r_0>0$ for some $r_0$ and every $k\geq 1$.

Then there exists $k_0\geq 1$ such that $\deg(V_T^{d_k},B(0,1))\neq 0$ for every $k\geq k_0$. In consequence, problem \eqref{problem1} has infinitely many solutions.
\end{thm}

\begin{proof}
Since $K$ is an ANR (see, e.g., \cite{bgp}), take $\eps>0$ such that each $\eps$-near maps from $S^{n-1}$ to $K$ are $r_0$-homotopic (see \cite{hu}, p. 111). Choose $d>0$ for $\eps$ as in Proposition \ref{prop-a}, and $k_0$ such that $d_{k_0}\geq d$. Then $\dist(V_T^{d_k}(S^{n-1}),0)>0$ for every $k\geq k_0$, and, by the homotopy property of the degree,
\[\deg(V_T^{d_k},B(0,1))=\deg(W_T^{d_k},B(0,1))\neq 0.\]
Moreover, by Theorem \ref{n-exist} (ii), there is a solution $x_k$ with an initial velocity $|v_k|\in (d_{k-1},d_k)$.
\end{proof}

\begin{ex}\label{ex-b}
\begin{em}
Take $K:=\overline{B(0,r)}\subset \rn$, where $n$ is odd, and $T:=1$. Then $W_T^{4k-3}=id_{S^{n-1}}$ and $W_T^{4k-1}=-id_{S^{n-1}}$. It implies that $\deg(W_T^{4k-3},B(0,1))=1$ and $\deg(W_T^{4k-1},B(0,1))=-1$. Moreover, $\dist(W_T^{2k-1}(S^{n-1}),0)=1$. Thus, all geometric assumptions from Theorem \ref{thm-a} are satisfied, and the theorem works for any integrably bounded right-hand side $f$.
\end{em}
\end{ex}

\noindent
{\bf Open problems and perspectives of research}
\renewcommand{\labelenumi}{\theenumi}
\begin{enumerate}
\item The author realizes that a computation of $\deg(V_T^{d},B(0,1))$ is difficult in general. He leaves as an open problem the question on sufficient conditions for $\deg(V_T^{d},B(0,1))\neq 0$. It will be a subject of a further study. Let us note that it is not hard to check that, if $\deg(W_T^d,B(0,1))\neq 0$ for some $d>0$, then for small $||m||_1$ one has $\deg(V_T^d,B(0,1))\neq 0$, as well. Hence problems with a uniform motion should be carefully examined.
\item Is it true that $\limsup_{d\to\infty}|\deg(V_T^{d},B(0,1))|>0$ for billiards which are strongly star-shaped with respect to $0$? Is it true for strictly convex billiards?
\item For simplicity we have assumed that the right-hand side $f$ implies the uniqueness of a solution of problem \eqref{problem-cauchy}. It is interesting and quite natural to consider multivalued right-hand sides because of the control theory, where we ask for a suitable control $u$ under which the Dirichlet problem
\[
 \left\{ \begin{array}{ll}
\xxx(t)=f(t,x(t),u(t)), & \mbox{for a.e. } t\in [0,T], x(t)\in int\, K,\\
\xx(s+)=\xx(s)+I(x(s),\xx(s)), & \mbox{if } x(s)\in \partial K,\\
x(0)=x(T)=0, &
\end{array}\right.
\]
has a solution. Indeed, we can define $F(t,x):=\{f(t,x,u) ; u\in U(t)\}$, where $U(t)$ is an admissible set of controls, and examine the inclusion $\xxx\in F(t,x)$.
\item It is worth checking the results presented in the paper under weaker growth conditions on the right-hand side $f$ (comp. assumption (H1)$_n$).
\item In the whole paper we have dealt with impulse functions satisfying condition (H0), that is, with fully elastic collisions with barrier. The non fully elastic case is left for the future study.
\end{enumerate}

\section*{Acknowledgements}

\noindent This research was supported by the Polish NCN grant no. 2013/09/B/ST1/01963.

\noindent The author thanks Mateusz Maciejewski for valuable numeric simulations helpful to create hypotheses.


\bibliographystyle{plain}
\bibliography{Dirichlet_billiards_GG}

\closegraphsfile
\end{document}